\documentclass[12pt,reqno]{amsart}

\usepackage{amsmath,amsthm,amssymb}
\usepackage{color,palatino}
\usepackage[colorlinks=true]{hyperref}
\hypersetup{citecolor=red}

\allowdisplaybreaks

\newcommand{\bex}{\begin{eqnarray*}}
\newcommand{\eex}{\end{eqnarray*}}
\newcommand{\be}{\begin{eqnarray}}
\newcommand{\ee}{\end{eqnarray}}

\newtheorem{thm}{\indent Theorem}

\newtheorem{lem}[thm]{\indent Lemma}
\newtheorem{pro}[thm]{\indent Proposition}

\newcommand{\C}{\mathbb{C}}

\newcommand{\inc}{\int_{\C}}

\newcommand{\dla}{d\lambda_\alpha}
\newcommand{\bn}{{\mathbb B}_n}
\newcommand{\cn}{{\mathbb C}^n}

\newtheorem*{tha}{\indent Theorem A}
\newtheorem*{thb}{\indent Theorem B}
\newtheorem*{thc}{\indent Theorem C}

\def\calD{\mathcal{D}}

\begin{document}

\title[Operators on the Fock space]
{Boundedness and Compactness of\\ Operators on the Fock space}

\author{Xiaofeng Wang}
\author{Guangfu Cao}
\author{Kehe Zhu}

\address{Wang and Cao: School of Mathematics and Information Science and
Key Laboratory of Mathematics and Interdisciplinary Sciences of the
Guangdong Higher Education Institute,
Guangzhou University, Guangzhou 510006, China}
\email{wangxiaofeng514@hotmail.com}
\email{guangfucao@163.com}

\address{Zhu: Department of mathematics and Statistics, State University of New York, 
Albany, NY 12222, USA}
\email{kzhu@math.albany.edu}



\thanks{Research of Wang and Cao supported by the China NNSF Grant 11271092.}

\begin{abstract}
We obtain sufficient conditions for a densely-defined operator on the Fock space to be 
bounded or compact. Under the boundedness condition we then characterize the compactness 
of the operator in terms of its Berezin transform.
\end{abstract}

\maketitle

\section{Introduction}

Let $\C$ be the complex plane and $\alpha$ be a positive parameter that is fixed 
throughout the paper. Let
$$d\lambda_{\alpha}(z)=\frac{\alpha}{\pi}e^{-\alpha|z|^2}\,dA(z)$$
be the Gaussian measure, where $dA$ is the Euclidean area measure. A calculation with 
polar coordinates shows that $d\lambda_{\alpha}$ is a probability measure.

The Fock space $F^2_{\alpha}$ consists of all entire functions $f$ in $L^2(\C,d\lambda_{\alpha})$. It is easy to show that $F^2_{\alpha}$ is a closed subspace
of $L^2(\C,d\lambda_{\alpha})$ and so is a Hilbert space with the inherited inner product
$$\langle f,g\rangle=\int_{\C}f(z)\overline{g(z)}d\lambda_{\alpha}(z).$$
In fact, $F^2_\alpha$ is a reproducing kernel Hilbert space whose kernel function is given by $$K_w(z)=K(z,w)=e^{\alpha z\overline{w}}.$$
The norm of functions in $L^2(\C,\dla)$ will simply be denoted by $\|f\|$. The norm of 
functions in $f\in L^p(\C,\dla)$ will be denoted by $\|f\|_p$.

We study linear operators (not necessarily bounded) on the Fock space. Throughout the
paper we let $\calD$ denote the set of all finite linear combinations $f$ of kernel functions
in $F^2_\alpha$:
$$f(z)=\sum_{k=1}^Nc_ke^{\alpha z\overline w_k}.$$
It is well known that $\calD$ is a dense linear subspace of $F^2_\alpha$. See \cite{Z3}
for example. We also assume that the domain of every linear operator that appears in the 
paper contains $\calD$. Using the relation $\langle SK_z,K_w\rangle=\langle K_z,S^*K_w\rangle$
we see that we can also assume that the domain of $S^*$ contains $\calD$ as well. One additional
standing assumption we make is that the function $z\mapsto SK_z$ is conjugate analytic.

Our main focus here is the boundedness and compactness of operators on $F^2_\alpha$.
To state our main results, we need to introduce a class of unitary operators on $F^2_\alpha$.
More specifically, for any $z\in\C$, let $\varphi_z$ denote the analytic self-map of $\C$
defined by $\varphi_z(w)=z-w$, let $k_z$ denote the normalized reproducing kernel
defined by
$$k_z(w)=K(w,z)/\sqrt{K(z,z)}=e^{-\frac\alpha2|z|^2+\alpha z\overline w},$$
and let $U_z$ denote the operator on $F^2_\alpha$ defined by $U_zf=f\circ\varphi_zk_z$.
Each $k_z$ is a unit vector in $F^2_\alpha$. It follows easily from a change of variables
that each $U_z$ is a self-adjoint unitary operator on $F^2_\alpha$. See \cite{Z3}.

For any $z\in\C$ and any linear operator $S$ on $F^2_\alpha$ let $S_z=U_zSU_z$. It is easy
to check that each $U_z$ maps $\calD$ onto $\calD$ (see Lemma~\ref{lem7}), so the domain 
of each $S_z$ contains $\calD$ whenever the domain of $S$ contains $\calD$. 

Each operator $S$ on $F^2_\alpha$ also induces a function $\widetilde S$ on $\C$, namely,
$$\widetilde S(z)=\langle Sk_z,k_z\rangle,\qquad z\in\C.$$
We call $\widetilde S$ the Berezin transform of $S$. Since each $k_z$ is a unit vector,
$\widetilde S$ is bounded whenever $S$ is bounded, and $\|\widetilde S\|_\infty\le\|S\|$.
Also, $k_z\to0$ weakly in $F^2_\alpha$ as $z\to\infty$, so $\widetilde S(z)\to0$ as
$z\to\infty$ whenever $S$ is compact on $F^2_\alpha$.

We can now state the main results of the paper.

\begin{tha}
If there exist some $p>2$ and $C>0$ such that $\|S_z1\|_p\le C$ for all $z\in\C$, then
the operator $S$ is bounded on $F^2_\alpha$.
\label{tha}
\end{tha}

\begin{thb}
If there exists some $p>2$ such that $\|S_z1\|_p\to0$ as $z\to\infty$, then $S$ is
compact on $F^2_\alpha$.
\label{thb}
\end{thb}

\begin{thc}
Suppose that there exist some $p>2$ and $C>0$ such that $\|S_z1\|_p\le C$ for all $z\in\C$.
Then $S$ is compact if and only if $\widetilde S(z)\to0$
as $z\to\infty$.
\label{thc}
\end{thc}

As an example, we will apply these results to the study of Toeplitz operators on 
$F^2_\alpha$.

The condition $\|S_z1\|_p\le C$ was first introduced in \cite{AZ} and
further studied in \cite{MZ}. An analogue of Theorem C was proved in \cite{MZ}
in the context of Bergman spaces on the unit disk. The papers \cite{CIL,DT,Zor} also
explore the condition $\|S_z1\|_p\le C$. 

Our approach here is different from those in the papers mentioned above, although
a key idea from \cite{AZ,MZ} will be used. One of the novelties here is that there is
no need for us to use Schur's test. 

A major difference exists between the Bergman space setting and the current one. 
More specifically, in the Bergman space setting, there is a certain cut-off requirement, 
namely, $p$ cannot be too close to $2$. In fact, it was shown in \cite{MZ} that $p$ must 
be greater than $3$ in the case of operators on the Bergman space of the unit disk. 
However, the cut-off requirement disappears in the Fock space setting; any $p>2$ will 
work. This is not entirely surprising; some similar situations were pointed out and 
explained in the book \cite{Z3}.

This work was done while the first named author was visiting the 
State University of New York at Albany. He wishes to thank the Department of Mathematics 
and Statistics at SUNY-Albnay for hosting his visit. The authors would also like to thank 
Josh Isralowitz and Haiying Li for helpful discussions.

\section{A sufficient condition for boundedness}

We prove Theorem A in this section. The following lemma will be used several times
in the paper.

\begin{lem}
For any $p>0$ we have
$$|f(z)|\le\left(\frac\beta\alpha\right)^{\frac1p}\|f\|_pe^{\frac\beta2|z|^2}$$
for all entire functions $f\in L^p(\C,\dla)$ and $z\in\C$, where $\beta=2\alpha/p$.
\label{lem1}
\end{lem}

\begin{proof}
It is clear that
$$\|f\|^p_p=\frac\alpha\pi\inc|f(z)|^pe^{-\alpha|z|^2}\,dA(z)=\frac\alpha\beta\cdot
\frac\beta\pi\inc\left|f(z)e^{-\frac\beta2|z|^2}\right|^p\,dA(z).$$
The desired estimate then follows from Corollary 2.8 in \cite{Z3}.
\end{proof}

We will also need the following estimate several times later on.

\begin{lem}
Suppose $p>2$ and $S$ is a linear operator on $F^2_\alpha$. Then
$$|\langle SK_w,K_z\rangle|\le\|S_w1\|_pe^{\frac\alpha2(|z|^2+|w|^2)-\sigma|z-w|^2}$$
for all $w$ and $z$ in the complex plane, where $\beta=2\alpha/p$ and $\sigma
=(\alpha-\beta)/2$. Consequently, if $\|S_w1\|_p\le C$ for some 
constant $C>0$ and all $w\in\C$, then for the same constant $C$ we have
$$|\langle SK_w,K_z\rangle|\le Ce^{\frac\alpha2(|z|^2+|w|^2)-\sigma|z-w|^2}$$
for all $z$ and $w$ in $\C$.
\label{lem2}
\end{lem}

\begin{proof}
Recall that
$$S_w1(z)=(U_wSU_w1)(z)=(U_wSk_w)(z)=k_w(z)(Sk_w)(w-z).$$
By Lemma~\ref{lem1}, we have
$$|k_w(z)(Sk_w)(w-z)|\le\left(\frac\beta\alpha\right)^{\frac1p}\,\|S_w1\|_p\,
e^{\frac\beta2|z|^2}\le\|S_w1\|_p\,e^{\frac\beta2|z|^2}$$
for all $z$ and $w$, where $\beta=2\alpha/p<\alpha$. Replacing $z$ by $w-z$, using
$$Sk_w(z)=e^{-\frac\alpha2|w|^2}SK_w(z)=e^{-\frac\alpha2|w|^2}\langle SK_w,K_z\rangle,$$
and simplifying the result, we obtain
$$|\langle SK_w,K_z\rangle|\le\|S_w1\|_p\,e^{\frac\alpha2|z|^2+\frac\alpha2|w|^2
-\sigma|z-w|^2}$$
for all $z$ and $w$.
\end{proof}

The following lemma shows that every linear operator on $F^2_\alpha$ can be represented
as an integral operator in a canonical way.

\begin{lem}
Let $S$ be a linear operator on $F^2_\alpha$ and let $T$ be the integral operator
defined on $L^2(\C,\dla)$ by
\begin{equation}
Tf(z)=\inc f(w)\langle SK_w,K_z\rangle\,\dla(w).
\label{eq1}
\end{equation}
Then $S$ is bounded on $F^2_\alpha$ if and only if $T$ is bounded on $L^2(\C,\dla)$.
Furthermore, when either of them is bounded, $S$ is equal to the restriction of $T$
to $F^2_\alpha$.
\label{lem3}
\end{lem}

\begin{proof}
For any fixed $z\in\C$, the function
$$w\mapsto\langle K_z,SK_w\rangle=\langle S^*K_z,K_w\rangle=(S^*K_z)(w)$$
is entire and belongs to $F^2_\alpha$ for any fixed $z\in\C$. Therefore, $Tf=0$ for every
$f\in L^2(\C,\dla)\ominus F^2_\alpha$. 

If $S$ is bounded on $F^2_\alpha$ and $f=K_a$ is the reproducing kernel at some point
$a\in\C$, then by the reproducing property of $K_a$,
\begin{eqnarray*}
Tf(z)&=&\inc K(w,a)\langle SK_w,K_z\rangle\,\dla(w)\\
&=&\overline{\inc\langle S^*K_z,K_w\rangle K(a,w)\,\dla(w)}\\
&=&\overline{\langle S^*K_z,K_a\rangle}=\langle SK_a,K_z\rangle\\
&=&SK_a(z)=Sf(z).
\end{eqnarray*}
It follows that $Tf=Sf$ on $\calD$ and $\|Tf\|\le\|S\|\|f\|$ for all $f\in\calD$. 
Combining this with the conclusion of the previous paragraph, we conclude that $T$ is
bounded on $L^2(\C,\dla)$ and $S$ is equal to the restriction of $T$ to $F^2_\alpha$.

Conversely, if $T$ is bounded on $L^2(\C,\dla)$ and $f\in\calD$, then
$$Tf(z)=\inc f(w)\overline{S^*K_z(w)}\,\dla(w)=\langle f, S^*K_z\rangle=
\langle Sf,K_z\rangle=Sf(z)$$
for all $z\in\C$. This shows that the restriction of $T$ on $\calD$ coincides with action
of $S$ there. Since $\calD$ is dense in $F^2_\alpha$ and $T$ is bounded, we conclude that
$S$ extends to a bounded linear operator on $F^2_\alpha$.
\end{proof}

We can now prove Theorem A which is the main result of this section.

\begin{thm}
Let $S$ be a linear operator on $F^2_\alpha$. If there are constants $p>2$ and $C>0$
such that $\|S_z1\|_p\le C$ for all $z\in\C$, then $S$ is bounded on $F^2_\alpha$ with
$\|S\|\le(2pC)/(p-2)$.
\label{th4}
\end{thm}

\begin{proof}
By Lemma~\ref{lem3}, it suffices for us to show that the integral operator $T$ defined
by (\ref{eq1}) is bounded on $L^2(\C,\dla)$.

By Lemma~\ref{lem2}, for the same constant $C$ and 
$$\sigma=\frac{\alpha-\beta}2=\frac{\alpha(p-2)}{2p},$$
we have
$$|Tf(z)|\le C\inc|f(w)|e^{\frac\alpha2(|z|^2+|w|^2)-\sigma|z-w|^2}\,d\lambda_\alpha(w)$$
for all $z\in\C$. Rewrite this as
$$F(z)\le C_1\inc|f(w)|e^{-\frac\alpha2|w|^2}e^{-\sigma|z-w|^2}\,dA(w),$$
where $C_1=C\alpha/\pi$ and
$$F(z)=|Tf(z)|e^{-\frac\alpha2|z|^2}.$$
By H\"older's inequality
\begin{eqnarray*}
F(z)^2&\le&C_1^2\inc\left|f(w)e^{-\frac\alpha2|w|^2}
\right|^2e^{-\sigma|z-w|^2}\,dA(w)\inc e^{-\sigma|z-w|^2}\,dA(w)\\
&=&C_2\inc\left|f(w)e^{-\frac\alpha2|w|^2}\right|^2e^{-\sigma|z-w|^2}\,dA(w),
\end{eqnarray*}
where
$$C_2=C_1^2\inc e^{-\sigma|u|^2}\,dA(u)=\frac{C_1^2\pi}\sigma.$$
It follows from Fubini's theorem and a change of variables that
$$\inc\left|Tf(z)e^{-\frac\alpha2|z|^2}\right|^2\,dA(z)\le C_3\inc
\left|f(w)e^{-\frac\alpha2|w|^2}\right|^2\,dA(w),$$
where
$$C_3=C_2\inc e^{-\sigma|u|^2}\,dA(u)=\frac{C_2\pi}\sigma
=\left(\frac{2pC}{p-2}\right)^2.$$
This shows that the operator $T$ is bounded on $L^2(\C,d\lambda_\alpha)$ and
$$\|T\|\le\frac{2p}{p-2}\,C.$$
Restricting $T$ to the space $F^2_\alpha$ then yields the desired result for $S$.
\end{proof}

Note that the proof above only depends on the pointwise estimate derived
in Lemma~\ref{lem2}, not the full assumption about the norms $\|S_z1\|_p$.

\section{Sufficient conditions for compactness}

In this section we present two sufficient conditions for an operator on $F^2_\alpha$
to be compact. The first condition is the little oh version of the condition in 
Theorem A, while the second condition is a natural deviation of the first one.

We begin with Theorem B, the companion result of Theorem A, which
we restate as follows.

\begin{thm}
Let $S$ be a linear operator on $F^2_\alpha$ and $p>2$. If $\|S_z1\|_p\to0$ as
$z\to\infty$, then $S$ is compact on $F^2_\alpha$.
\label{th5}
\end{thm}

\begin{proof}
It follows from our standing assumptions on $S$ that the condition $\|S_z1\|_p\to0$ 
as $z\to\infty$ implies that $\|S_z1\|_p$ is bounded in $z$. Therefore, by Theorem~\ref{th4},
$S$ is already bounded on $F^2_\alpha$. By Lemma~\ref{lem3}, it suffices for us to show 
that the integral operator $T$ defined by (\ref{eq1}) is compact on $L^2(\C,\dla)$. 
We do this using an approximation argument.

For any $r>0$ let us consider the integral operator $T_r$ defined on $L^2(\C,\dla)$ by
\begin{eqnarray*}
T_rf(z)&=&\int_{|w|<r}f(w)\langle SK_w,K_z\rangle\,\dla(w)\\
&=&\inc f(w)\chi_r(w)\langle SK_w,K_z\rangle\,\dla(w),
\end{eqnarray*}
where $\chi_r$ is the characteristic function of the disk $\{z\in\C:|z|<r\}$.
It follows easily from Lemma~\ref{lem2} that
$$\inc\inc|\chi_r(w)\langle SK_w,K_z\rangle|^2\,\dla(z)\,\dla(w)<\infty.$$
Thus each $T_r$ is Hilbert-Schmidt. In particular, each $T_r$ is compact on 
$L^2(\C,\dla)$.

Let $D_r=T-T_r$. Then
\begin{eqnarray*}
D_rf(z)&=&\inc f(w)(1-\chi_r(w))\langle SK_w,K_z\rangle\,\dla(w)\\
&=&\int_{|w|>r}f(w)\langle SK_w,K_z\rangle\,\dla(w).
\end{eqnarray*}
We are going to show that $\|D_r\|\to0$ as $r\to\infty$, which would imply that $T$ is 
compact.

Given any $\varepsilon>0$, choose a positive number $R$ such that $\|S_w1\|_p<
\varepsilon$ for all $|w|>R$. By Lemma~\ref{lem2}, for any $r>R$ we have
$$|1-\chi_r(w)||\langle SK_w,K_z\rangle|\le\varepsilon e^{\frac\alpha2(|z|^2+|w|^2)
-\sigma|z-w|^2}$$
for all $z$ and $w$ in $\C$ (just consider the cases $|w|\le r$ and $|w|>r$ separately). 
It follows from the proof of Theorem~\ref{th4} that there is
a positive constant $C$, independent of $\varepsilon$ and $r$, such that
$\|D_r\|\le C\varepsilon$ for all $r>R$. This shows that $\|D_r\|\to0$ as $r\to\infty$
and completes the proof of the theorem.
\end{proof}

Recall from the definition of $S_z$ and $U_z$ that
$$S_z1=U_zSU_z1=U_zSk_z,\qquad z\in\C.$$
Since each $U_z$ is a unitary operator on $F^2_\alpha$, the condition $\|S_z1\|\le C$
is the same as $\|Sk_z\|\le C$. However, $U_z$ is not isometric on $L^p(\C,\dla)$ when
$p\not=2$, so it is natural for us to consider the condition $\|Sk_z\|_p\le C$.

\begin{pro}
Let $S$ be a linear operator on $F^2_\alpha$ and $p>2$. If there is a constant $C>0$
such that $\|Sk_z\|_p\le C$ and $\|S^*k_z\|_p\le C$ for all $z\in\C$, then $S$ is 
Hilbert-Schmidt on $F^2_\alpha$. In particular, $S$ is compact.
\label{pro6}
\end{pro}

\begin{proof}
By Lemma~\ref{lem1}, the assumption on $\|Sk_w\|_p$ implies that there exists another
positive constant $C$ such that
$$|(Sk_w)(z)|\le Ce^{\frac\beta2|z|^2},\qquad z,w\in\C,$$
where $\beta=2\alpha/p<\alpha$. This can be rewritten as
\begin{equation}
|\langle SK_w,K_z\rangle|\le Ce^{\frac\alpha2|w|^2+\frac\beta2|z|^2}
\label{eq2}
\end{equation}
for all $z$ and $w$. Similarly, the assumption on $\|S^*k_w\|_p$ implies that
\begin{equation}
|\langle SK_w,K_z\rangle|\le Ce^{\frac\alpha2|z|^2+\frac\beta2|w|^2}
\label{eq3}
\end{equation}
for all $z$ and $w$.

Multiply the inequalities in (\ref{eq2}) and (\ref{eq3}) and then take the square 
root on both sides. The result is
$$|\langle SK_w,K_z\rangle|\le Ce^{\frac\delta2(|z|^2+|w|^2)}$$
for all $z$ and $w$, where $\delta=(\alpha+\beta)/2<\alpha/2$. It follows from this that
$$\inc\inc\left|\langle SK_w,K_z\rangle\right|^2\,\dla(w)\,\dla(z)<\infty,$$
so that the integral operator $T$ defined by
$$Tf(z)=\inc f(w)\langle SK_w,K_z\rangle\,d\lambda_\alpha(w)$$
is Hilbert-Schmidt on $L^2(\C,d\lambda_\alpha)$. Since $S$ is the restriction of $T$ on
$F^2_\alpha$, we conclude that $S$ is Hilbert-Schmidt on $F^2_\alpha$.
\end{proof}

Once again, we only used the pointwise estimates deduced from the assumptions on
$\|Sk_z\|_p$ and $\|S^*k_z\|_p$.

\section{Compactness via the Berezin transform}

In this section we show that, under the assumption of Theorem A, the 
compactness of a linear operator on $F^2_\alpha$ can be characterized in terms of 
its Berezin transform.

\begin{lem}
For any $a$ and $w$ in the complex plane we have
$$U_aK_w=\overline{k_a(w)}K_{\varphi_a(w)},\quad U_ak_w=\beta k_{\varphi_a(w)},
\quad \widetilde S\circ\varphi_a=\widetilde{S_a},$$
where $\beta$ is a unimodular constant depending on $a$ and $w$.
\label{lem7}
\end{lem}

\begin{proof}
The first identity follows from the definition of $U_a$ and the explicit form of the
kernel function. The second identity follows from the first one with 
$$\beta=e^{\frac\alpha2(a\overline w-\overline aw)}.$$
By the definition of the Berezin transform, the definition of $S_a$, and the second
identity that we have already proved, we have
\begin{eqnarray*}
\widetilde{S_a}(w)&=&\langle S_ak_w,k_w\rangle=\langle U_aSU_ak_w,k_w\rangle\\
&=&\langle SU_ak_w,U_ak_w\rangle=|\beta|^2\langle Sk_{\varphi_a(w)},k_{\varphi_a(w)}\rangle\\
&=&\widetilde S(\varphi_a(w)).
\end{eqnarray*}
This proves the third identity.
\end{proof}

\begin{lem}
Let $S$ be a linear operator on $F^2_{\alpha}$. Suppose that there are constants $p>2$ 
and $C>0$ such that $\|S_z1\|_p\le C$ for all $z\in\C$. Then $\widetilde{S}(w)\to0$ 
as $w\to\infty$ if and only if for every (or some) $2<p'<p$ we have $\|S_w 1\|_{p'}\to0$ 
as $w\to\infty$.
\label{lem8}
\end{lem}

\begin{proof}
If for some $p'\in (2,p)$ we have $\|S_w 1\|_{p'}\rightarrow 0$ as $w\rightarrow \infty$,
then by H\"older's inequality,
$$|\widetilde{S}(w)|=|\langle S_w 1, 1\rangle| \leq \|S_w 1\|_{p'}\rightarrow 0$$
as $w\rightarrow \infty$.

Next, suppose $\widetilde{S}(w)\rightarrow 0$ as $w\rightarrow \infty$ and fix any
$p'\in (2,p)$. We proceed to show that $\|S_w 1\|_{p'}\rightarrow 0$ as $w\rightarrow\infty$.

For any $a$ and $z$ we have
$$\widetilde{S}(\varphi_z(a))=\widetilde{S_z}(a)=e^{-\alpha|a|^2}\langle S_z K_a,K_a \rangle,$$
where
$$K_a(u)=e^{\alpha u\overline{a}}=\sum_{k=0}^{\infty}\frac{\alpha^k}{k!}u^k \overline{a}^k.$$
By the proof of Lemma 6.26 in \cite{Z3}, starting at line 4 from the bottom of page 240 and
finishing at line 3 from the top of page 242, with $\widetilde f$ replaced by $\widetilde S$ and
$T_{f\circ\varphi_z}$ replaced by $S_z$, we will have
$$\lim_{z\to\infty}\langle S_z1, z^n\rangle=0$$
for every $n\ge0$. Since the polynomials are dense in $F^2_\alpha$, we conclude that
$S_z1\to 0$ weakly in $F^2_\alpha$ as $z\to\infty$. In particular, for every $w\in\C$,
$S_z1(w)\to 0$ as $z\to\infty$.

Let $s=p/p'>1$ and choose $t>1$ such that $1/s+1/t=1$. For any measurable set $E$ we have
\begin{eqnarray*}
\int_E|S_z1(w)|^{p'}\,d\lambda_\alpha(w)&\le&\left[\int_E|S_z1(w)|^p\,d\lambda_\alpha(w)
\right]^{\frac1s}\left[\int_E d\lambda_\alpha(w)\right]^{\frac1t}\\
&\le&\|S_z1\|_p^{p'}\left[\lambda_\alpha(E)\right]^{\frac1t}.
\end{eqnarray*}
Since $\|S_z1\|_p\le C$ for all $z\in\C$, this shows that the family $\{|S_z1|^{p'}:z\in\C\}$
is uniformly integrable. By Vitali's Theorem, 
$$\lim_{z\to\infty}\inc|S_z1(w)|^{p'}\,d\lambda_\alpha(w)=0.$$
This completes the proof of the lemma.
\end{proof}

We can now prove Theorem C, the main result of this section, which we restate as follows.

\begin{thm}
Suppose $S$ is a linear operator on $F^2_\alpha$, $p>2$, $C>0$, and $\|S_z1\|_p\le C$
for all $z\in\C$. Then $S$ is compact on $F^2_\alpha$ if and only if $\widetilde S(z)
\to0$ as $z\to\infty$.
\label{th9}
\end{thm}

\begin{proof}
By Theorem~\ref{th4}, $S$ is bounded on $F^2_\alpha$. If $S$ is further compact, then
$\widetilde S(z)=\langle Sk_z,k_z\rangle\to0$ as $z\to\infty$, because $k_z\to0$ weakly
in $F^2_\alpha$ as $z\to\infty$.

Conversely, if $\widetilde S(z)\to0$ as $z\to\infty$, it follows from Lemma~\ref{lem8} 
that $\|S_z1\|_{p'}\to0$ as $z\to\infty$, where $p'$ is any fixed number strictly between
$2$ and $p$. This together with Theorem~\ref{th5} then implies that $S$ is compact.
\end{proof}

\section{An application to Toeplitz operators}

Let $P:L^2(\C,\dla)\to F^2_\alpha$ denote the orthogonal projection. If 
$\psi\in L^\infty(\C)$, we can define a linear operator $T_\psi$ on $F^2_\alpha$ by
$T_\psi f=P(\psi f)$. It is clear that $T_\psi$ is bounded and $\|T_\psi\|\le\|\psi\|_\infty$.
It is also easy to verifty that 
$$(T_\psi)_z=U_zT_\psi U_z=T_{\psi\circ\varphi_z}$$
for all $z\in\C$. In particular, $(T_\psi)_z1=P(\psi\circ\varphi_z)$, or
$$(T_\psi)_z1(w)=\inc K(w,u)\psi(z-u)\,\dla(u),\qquad w\in\C.$$
It follows that
\begin{eqnarray*}
|(T_\psi)_z1(w)|&\le&\|\psi\|_\infty\inc|e^{\alpha w\overline u}|\,\dla(u)\\
&=&\|\psi\|_\infty\inc\left|e^{\alpha(w/2)\overline u}\right|^2\,\dla(u)\\
&=&\|\psi\|_\infty e^{\alpha|w/2|^2}=\|\psi\|_\infty e^{\frac\alpha4|w|^2}
\end{eqnarray*}
for all $w\in\C$. This shows that
$$\sup_{z\in\C}\inc\left|(T_\psi)_z1\right|^p\,\dla<\infty$$
whenever $0<p<4$. Therefore, the assumption in Theorem~\ref{th9} is satisfied for each
$p\in(2,4)$. Consequently, we arrive at the well-known result that such a Toeplitz operator 
is compact if and only if its Berezin transform vanishes at $\infty$. See \cite{CIL,Z3}.

Using the integral representation for the orthogonal projection, it is possible to
define Toeplitz operators $T_\psi$ for functions $\psi$ that are not necessarily
bounded. In particular, $T_\psi$ is well defined on $\calD$ whenever $\psi$ belongs
to the space BMO used in Section 6.4 of \cite{Z3}. For such a symbol function $\psi$,
Lemma 6.25 of \cite{Z3} states that
$$|(T_\psi)_z1(w)|\le Ce^{\frac\alpha4|w|^2},\qquad w\in\C,$$
whenever $\widetilde\psi$ is bounded, which implies that $\|(T_\psi)_z1\|_p\le C$ for
$2<p<4$. This together with the arguments in the previous 
paragraph shows that, for such $\psi$, the operator $T_\psi$ is bounded if and only
if its Berezin transform is bounded; and $T_\psi$ is compact if and only if its
Berezin transform vanishes at $\infty$. See \cite{CIL,Z3} again.

The Berezin transform of $T_\psi$ is usually written as
$\widetilde\psi$ or $B_\alpha\psi$. It is easy to see that
$$B_\alpha\psi(z)=\inc\psi(z-w)\,\dla(w)=\frac\alpha\pi\inc\psi(w)e^{-\alpha|z-w|^2}
\,\dla(w)$$
for $z\in\C$. See \cite{Z3} for more information about the Berezin transform which is
also called the heat transform in many articles.

The arguments above can also be extended to operators of the form 
$T=T_{\psi_1}\cdots T_{\psi_n}$, where each $\psi_k$ belongs to $L^\infty(\C)$. 
In fact, in the case $S=T_{\psi_1}T_{\psi_2}$, we have
$$S_z1=P\left[\psi_1\circ\varphi_z(T_{\psi_2})_z1\right].$$
Using the integral representation for the outside $P$ and the pointwise estimate we
already obtained for $(T_{\psi_2})_z1$, we arrive at
$$|S_z1(w)|\le C\inc e^{\frac\alpha4|u|^2}|K(w,u)|\,\dla(u)\le C_1e^{\frac\alpha3|w|^2}.$$
This implies that 
$$\sup_z\|S_z1\|_p<\infty,\qquad 2<p<3.$$
More generally, if $\sigma>2$, then
$$\inc e^{\frac\alpha\sigma|u|^2}|K(w,u)|\,\dla(u)\le Ce^{\frac\alpha{\sigma'}|w|^2},$$
with
$$\sigma'=4\left(1-\frac1\sigma\right)>2.$$
So by mathematical induction, each operator $S=T_{\psi_1}\cdots T_{\psi_n}$ satisfies
the pointwise estimate
$$|S_z1(w)|\le Ce^{\frac\alpha\sigma|w|^2},\qquad w\in\C,$$
for some $\sigma>2$. It follows that 
$$\sup_z\|S_z1\|_p<\infty,\qquad p\in(2,\sigma).$$

Going one step further, we can also extend the arguments above to operators on 
$F^2_\alpha$ that are finite sums of finite products of Toeplitz operators.

\section{Further results and remarks}

For any $p>0$ the Fock space $F^p_\alpha$ is defined to be the set of all entire 
functions $f$ such that $f(z)e^{-\frac\alpha2|z|^2}$ belongs to $L^p(\C,dA)$. The
norm in $F^p_\alpha$ is defined by
$$\|f\|^p_{p,\alpha}=\frac{p\alpha}{2\pi}\inc\left|f(z)e^{-\frac\alpha2|z|^2}
\right|^p\,dA(z).$$
It is clear that when $p=2$, the definition here is consistent with the definition
of $F^2_\alpha$ in the Introduction. More generally, we have
$$F^p_\alpha=H(\C)\cap L^p(\C,d\lambda_\beta),\qquad \beta=\frac{p\alpha}2,$$
where $H(\C)$ is the space of all entire functions. Equivalently,
$$H(\C)\cap L^p(\C,\dla)=F^p_\beta,\qquad \beta=\frac{2\alpha}p.$$

Although both $F^p_\alpha$ and $H(\C)\cap L^p(\C,\dla)$ are natural extentions of
the Fock space $F^2_\alpha$, in most cases it is much more beneficial, more 
convenient, and more natural to use $F^p_\alpha$ instead of the other one. Of course
there are exceptions, the results of this paper being one of them. Nevertheless,
the following question still seems natural: What happens if we replaced the 
condition $\|S_z1\|_p\le C$ by the condition $\|S_z1\|_{p,\alpha}\le C$? We do not 
know the answer. But the techniques used in the paper would certainly not work,
because the optimal pointwise estimate for functions in $F^p_\alpha$ is given by
$$|f(z)|\le\|f\|_{p,\alpha}e^{\frac\alpha2|z|^2},\qquad z\in\C.$$
See Corollary 2.8 in \cite{Z3}. We needed a certain decrease in the exponent in
order to perform the analysis in Sections 2--4.

In the case of $S=T_\psi$, where $\psi\in L^\infty(\C)$, we already showed that 
$$\sup_{z\in\C}\|(T_\psi)_z1\|_p<\infty,\qquad \sup_{z\in\C}\|(T_\psi^*)_z1\|_p<\infty,$$
for $0<p<4$. On the other hand, for every $p\in[1,\infty)$, the projection $P$ is 
bounded from the space
$$L^p_\alpha(\C)=\left\{f:f(z)e^{-\frac\alpha2|z|^2}\in L^p(\C,dA)\right\}$$
onto the space $F^p_\alpha$; see \cite{Z3} for example. It follows from this and the
identity $(T_\psi)_z1=P(\psi\circ\varphi_z)$ that
$$\sup_{z\in\C}\|(T_\psi)_z1\|_{p,\alpha}<\infty,\qquad
\sup_{z\in\C}\|(T_\psi^*)_z1\|_{p,\alpha}<\infty,$$
for all $1\le p<\infty$. Thus the condition $\|S_z1\|_p\le C$ appears stronger (or
more difficult to satisfy) than the condition $\|S_z1\|_{p,\alpha}\le C$. This is 
easily confirmed by the elementary continuous embedding
$$H(\C)\cap L^p(\C,\dla)=F^p_\beta\subset F^p_\alpha,$$
where $\beta=(2\alpha)/p<\alpha$ for $p>2$. 

The example in the previous section of Toeplitz operators on $F^2_\alpha$ induced by 
bounded symbols shows that the condition $\|S_z1\|_p\le C$ is a meaningful one. We just 
do not know what the weaker condition $\|S_z1\|_{p,\alpha}\le C$ would imply. But there 
is more we can say.

For each $z\in\C$ the operator $U_z$ is actually a surjective isometry on each 
$F^p_\alpha$, and $k_z$ is actually a unit vector in $F^p_\alpha$. Therefore, the 
condition $\|S_z1\|_{p,\alpha}\le C$ is the same as $\|Sk_z\|_{p,\alpha}\le C$.
If there exists a bounded linear operator $S$ on $F^p_\alpha$, $2<p<\infty$, such
that $S$ is not bounded on $F^2_\alpha$, then the condition $\|S_z1\|_{p,\alpha}\le C$ 
would not imply the boundedness of $S$ on $F^2_\alpha$. Although we do not have an
example at hand, this seems very plausible to us.

Note that the proof of Theorem~\ref{th4} amounts to showing that the integral 
operator $T$ defined by
$$Tf(z)=\inc f(w)H(z,w)\,d\lambda_\alpha(w)$$
is bounded on $L^2(\C,d\lambda_\alpha)$, where 
$$H(z,w)=e^{\frac\alpha2(|z|^2+|w|^2)-\sigma|z-w|^2}.$$ 
Since $f\in L^2(\C,d\lambda_\alpha)$ if and only if the function
$f(w)e^{-\frac\alpha2|w|^2}$ is in $L^2(\C,dA)$, and since
$$e^{-\frac\alpha2|z|^2}Tf(z)=\frac\alpha\pi\inc\left[f(w)e^{-\frac\alpha2|w|^2}\right]
e^{-\sigma|z-w|^2}\,dA(w),$$
we see that the operator $T$ on $L^2(\C,d\lambda_\alpha)$ is unitarily equivalent to 
the Berezin transform $B_\sigma$ as an operator on $L^2(\C,dA)$. Recall that
$$B_\sigma f(z)=\frac\sigma\pi\inc f(w)e^{-\sigma|z-w|^2}\,dA(w).$$
The boundedness of $B_\sigma$ on $L^2(\C,dA)$ is actually a known result. 
See \cite{Z3} for example.

A natural question here is the following: is the Berezin transform $B_\sigma$
compact on $L^2(\C,dA)$? Since the proof of Theorem~\ref{th4} along with the fact that
$\|(T_\psi)_z1\|_p\le C$ for $2<p<4$ shows that every
Toeplitz operator $T_\psi$ on $F^2_\alpha$, $\psi\in L^\infty(\C)$, is dominated by 
$B_\sigma$ as an operator on $L^2(\C,dA)$, and it is very easy to see that there are 
such Toeplitz operators that are not compact, we see that $B_\sigma$ cannot possibly 
be compact on $L^2(\C,dA)$. To see this more directly, we consider the sequence 
$\{\chi_n\}$ of characteristic functions of the disks $B(n,1)$. It is easy to see 
that $\{\chi_n\}$ converges to $0$ weakly in $L^2(\C,dA)$. But
$$B_\sigma\chi_n(z)=\frac\sigma\pi\int_{B(0,1)}e^{-\sigma|z-n-w|^2}\,dA(w)=g(z-n),$$
where
$$g(z)=\frac\sigma\pi\int_{B(0,1)}e^{-\sigma|z-w|^2}\,dA(w).$$
By translation invariance, the norm of each $B_\sigma\chi_n$ in $L^2(\C,dA)$ is equal 
to that of $g$. Thus $\|B_\sigma\chi_n\|_{L^2(\C,dA)}\not\to0$ as $n\to\infty$, so 
$B_\sigma$ is not compact on $L^2(\C,dA)$.

Our arguments can also be adapted to work for Bergman spaces on the unit ball $\bn$ in
$\cn$. More specifically, for any $\alpha>-1$ we consider the weighted volume measure
$$dv_\alpha(z)=c_\alpha(1-|z|^2)^\alpha\,dv(z),$$
where $dv$ is ordinary volume measure on $\bn$ and $c_\alpha$ is a normalizing constant
chosen so that $v_\alpha(\bn)=1$. For any $p>0$ the spaces
$$A^p_\alpha=H(\bn)\cap L^p(\bn,dv_\alpha)$$
are called (weighted) Bergman spaces, where $H(\bn)$ is the space of all holomorphic
functions on $\bn$.

The space $A^2_\alpha$ is a reproducing kernel Hilbert space whose reproducing kernel
is given by
$$K(z,w)=\frac1{(1-\langle z,w\rangle)^{n+1+\alpha}}.$$
The normalized reproducing kernels are still defined by
$$k_z(w)=\frac{K(w,z)}{\sqrt{K(z,z)}}=\frac{(1-|z|^2)^{\frac{n+1+\alpha}2}}{(1-
\langle z,w\rangle)^{n+1+\alpha}}.$$

For every $z\in\bn$ there is also a canonical involutive automorphism $\varphi_z$ of
the unit ball $\bn$, and an associated self-adjoint unitary operator $U_z$ can be defined
on $A^2_\alpha$ by $U_zf=f\circ\varphi_zk_z$. If $S$ is a linear operator on $A^2_\alpha$,
not necessirly bounded, whose domain contains all finite linear combinations of kernel
functions, then we can still consider $S_z=U_zSU_z$.

The optimal pointwise estimate for functions in Bergman spaces is given by
$$|f(z)|\le\frac{\|f\|_{A^p_\alpha}}{(1-|z|^2)^{\frac{n+1+\alpha}p}}.$$
See \cite{Z2} for this and the results quoted in the previous two paragraphs. It 
follows from the proof of Lemma~\ref{lem2} that the condition
$$\sup_{z\in\bn}\|S_z1\|_{A^p_\alpha}<\infty,$$
where $p>2$, implies the inequality
$$|\langle SK_w,K_z\rangle|\le\frac{C|1-\langle z,w\rangle|^{\left(\frac2p
-1\right)(n+1+\alpha)}}{(1-|z|^2)^{\frac{n+1+\alpha}p}(1-|w|^2)^{\frac{n+1+\alpha}p}}.$$
Our techniques here can be adapted to show that for 
\begin{equation}
p>2+\frac{2n}{\alpha+1},
\label{eq4}
\end{equation}
the condition $\|S_z1\|_{A^p_\alpha}\le C$ implies that the operator $S$ is bounded
on $A^2_\alpha$. Similarly, the condition
$$\lim_{|z|\to1^-}\|S_z1\|_{A^p_\alpha}=0$$
implies that the operator $S$ is not only bounded but also compact on $A^2_\alpha$.
Furthermore, under the assumption $\|S_z1\|_{A^p_\alpha}\le C$, the compactness of
$S$ on $A^2_\alpha$ is equivalent to the vanishing of the Berezin transform of $S$
on the unit sphere $|z|=1$. We leave the details to the interested reader.

We point out that in the case when $n=1$ and $\alpha=0$, 
the restriction $p>4$ in (\ref{eq4}) is not as good as the optimal restriction $p>3$ 
obtained in \cite{MZ}. The discrepancy stems from the fact that our approach here only
uses pointwise estimates derived from the assumption about norms, while the approach in
\cite{MZ} made full use of the assumption about norms.

We also mention that the conditions
\begin{equation}
\sup_{z\in\bn}\|Sk_z\|_{A^p_\alpha}<\infty,\quad
\sup_{z\in\bn}\|S^*k_z\|_{A^p_\alpha}<\infty,
\label{eq5}
\end{equation}
where $p>2$, imply the inequality
$$|\langle SK_w,K_z\rangle|\le\frac C{(1-|z|^2)^{\frac{n+1+\alpha}{q}}
(1-|w|^2)^{\frac{n+1+\alpha}{q}}},$$
where $q\in(2,p)$ is the exponent given by
$$\frac1{q}=\frac12\left(\frac12+\frac1p\right).$$
If $p$ and $\alpha$ satisfy
$$\alpha+1>\left(\frac12+\frac1p\right)(n+1+\alpha),$$
then the conditions in (\ref{eq5}) imply that $S$ is Hilbert-Schmidt on
$A^2_\alpha$. Obviously, the dependence on $p$ and $\alpha$ in the Bergman space
theory is much more delicate. Again, the interested reader can easily work out the
details by following arguments in previous sections of this paper.

\end{document}